\documentclass[reqno, 12pt,a4letter]{amsart}
\usepackage{amsmath,amsxtra,amssymb,latexsym, amscd,amsthm}
\usepackage{graphicx,color}
\usepackage[utf8]{inputenc}
\usepackage[mathscr]{euscript}
\usepackage{mathrsfs}
\usepackage{eqname}
\usepackage[english]{babel}
\usepackage{color}
\newtheorem {theorem}{Theorem}[section]

\newtheorem {lemma}{Lemma}[section]

\newtheorem {example}{Example}[section]
\newtheorem {definition}{Definition}[section]

\newtheorem {remark}{Remark}[section]

\def\ar{a\kern-.370em\raise.16ex\hbox{\char95\kern-0.53ex\char'47}\kern.05em}
\def\ees{{\accent"5E e}\kern-.385em\raise.2ex\hbox{\char'23}\kern-.08em}
\def\eex{{\accent"5E e}\kern-.470em\raise.3ex\hbox{\char'176}}
\def\AR{A\kern-.46em\raise.80ex\hbox{\char95\kern-0.53ex\char'47}\kern.13em}
\def\EES{{\accent"5E E}\kern-.5em\raise.8ex\hbox{\char'23 }}
\def\EEX{{\accent"5E E}\kern-.60em\raise.9ex\hbox{\char'176}\kern.1em}
\def\ow{o\kern-.42em\raise.82ex\hbox{
  \vrule width .12em height .0ex depth .075ex \kern-0.16em \char'56}\kern-.07em}
\def\OW{O\kern-.460em\raise1.36ex\hbox{
\vrule width .13em height .0ex depth .075ex \kern-0.16em \char'56}\kern-.07em}
\def\UW{U\kern-.42em\raise1.36ex\hbox{
\vrule width .13em height .0ex depth .075ex \kern-0.16em \char'56}\kern-.07em}
\def\DD{D\kern-.7em\raise0.4ex\hbox{\char '55}\kern.33em}

\title{Invariants of the bi-Lipschitz contact equivalence of continuous definable function germs}

\author{TI\EES N-S\OW N PH\d{A}M}
\address{Department of Mathematics, University of Dalat, 1 Phu Dong Thien Vuong, Dalat, Vietnam}
\email{sonpt@dlu.edu.vn}

\author{NGUY\EEX N TH\AR O NGUY\^EN B\`UI}
\address{Department of Pedagogy, University of Dalat, 1 Phu Dong Thien Vuong, Dalat, Vietnam}
\email{nguyenbnt@dlu.edu.vn}

\date{\today}


\thanks{The authors are partially supported by Vietnam National Foundation for Science and Technology Development (NAFOSTED), grant 101.04-2016.05}
\begin{document}

\begin{abstract}
We construct an invariant of the bi-Lipschitz contact equivalence of continuous function germs definable in a polynomially bounded o-minimal structure, such as semialgebraic functions. For a germ $f,$ the invariant is given in terms of the leading coefficients of the asymptotic expansions of $f$ along the connected components of the tangency variety of $f.$
\end{abstract}

\maketitle

\section{Introduction}

Lipschitz geometry of maps is a rapidly growing subject in contemporary Singularity Theory. Recent progress in this area is due to the tameness theorems proved by
several authors (see, for example, \cite{Birbrair2007, Fukuda1976, Henry2003, Henry2004, Kuo1985, Ruas2011}). However the description of a set of invariants is barely developed (see also \cite{Birbrair2017}). This paper presents a numerical  invariant of continuous function germs definable in a polynomially bounded o-minimal structure (e.g., semialgebraic functions) with respect to the bi-Lipschitz contact equivalence. 
The most important ingredient of the invariant constructed here is the so-called tangency variety. More precisely, let $f \colon (\mathbb{R}^n, 0) \rightarrow ({\Bbb R}, 0)$ be a continuous function germ, which is definable in a polynomially bounded o-minimal structure. The tangency variety $\Gamma(f)$ of $f$ consists of all points $x$ in some neighborhood of the origin $0 \in \mathbb{R}^n$ such that the fiber $f^{-1}(f(x))$ is tangent to the sphere in $\mathbb{R}^n$ centered at $0$ with radius $\|x\|.$ The restriction of $f$ on each connected component of $\Gamma(f) \setminus \{0\}$ defines a definable function $f_k$ of a single variable. Then the invariant of $f$ is given in terms of the leading coefficients of the asymptotic expansions of these functions $f_k.$


The rest of the paper is organized as follows. In Section~\ref{SectionPreliminary}, we present some preliminaries which will be used later. 
The definition and some properties of tangency varieties are given in Section~\ref{Section3}. The main result is provided in Section~\ref{Section4}. 

\section{Preliminaries}\label{SectionPreliminary}

Throughout this work we shall consider the Euclidean vector space ${\Bbb R}^n$ endowed with its canonical scalar product $\langle \cdot, \cdot \rangle$ and we shall denote its associated norm $\| \cdot \|.$ The closed ball (resp., the sphere) centered at the origin $0 \in \mathbb{R}^n$ of radius $\epsilon$ will be denoted by $\mathbb{B}_{\epsilon}$ (resp., $\mathbb{S}_{\epsilon}$).


\subsection{The bi-Lipschitz contact equivalence}

The contact equivalence between (smooth) mappings was introduced by J.~Mather \cite{Mather1968}. The natural extension of Mather's definition to the Lipschitz setting in the function case appeared in \cite{Birbrair2007}, and to the general case in \cite{Ruas2011}. Let us start with the following definition. 

\begin{definition}{\rm
Two map germs $f, g \colon (\mathbb{R}^n, 0) \to (\mathbb{R}^p,0)$ are called {\em bi-Lipschitz contact equivalent} (or {\em $\mathcal{K}$-bi-Lipschitz equivalent}) if there exist two germs of bi-Lipschitz homeomorphisms $h \colon (\mathbb{R}^n ,0) \to (\mathbb{R}^n,0)$ and $H \colon (\mathbb{R}^n \times \mathbb{R}^p,0) \to (\mathbb{R}^n \times \mathbb{R}^p, 0)$ such that $H (\mathbb{R}^n \times \{ 0\}) = \mathbb{R}^n \times \{ 0\} $ and the following diagram is commutative:
$$ \CD  (\mathbb{R}^n ,0) @>(id, f)>> (\mathbb{R}^n \times \mathbb{R}^p ,0)  @>\pi_n>> (\mathbb{R}^n ,0) \\ @Vh VV @VVHV @Vh VV \\  (\mathbb{R}^n ,0)   @ >(id, g)>> (\mathbb{R}^n \times \mathbb{R}^p ,0) @>\pi_n>> (\mathbb{R}^n ,0)
\endCD$$
where $id \colon \mathbb{R}^n \to \mathbb{R}^n$ is the identity map and $\pi_n \colon \mathbb{R}^n \times \mathbb{R}^p \to \mathbb{R}^n$ is the canonical projection.
}\end{definition}

In this paper we consider the case $p = 1,$ thus the maps $f, g$ are functions. There is a more convenient way to work with the bi-Lipschitz contact equivalence of functions, due to the following result:

\begin{theorem}[{see \cite[Theorem 2.1]{Birbrair2007}}]\label{Theorem1}
Let $f, g \colon (\mathbb{R}^n ,0) \to (\mathbb{R},0)$ be two continuous function germs. If $f$ and $g$ are bi-Lipschitz contact equivalent, then 
there exists a bi-Lipschitz homeomorphism germ $h \colon (\mathbb{R}^n ,0) \to (\mathbb{R}^n,0),$ there exist positive constants $c_1, c_2$ and a sign $\sigma \in \{-1, 1\}$ such that in a neighbourhood of the origin $0 \in \mathbb{R}^n$ the following inequalities hold true
\begin{eqnarray*}
c_1  f(x)  & \leq & \sigma g(h(x)) \  \leq \ c_2  f(x).
\end{eqnarray*}
\end{theorem}

\subsection{O-minimal structures}

The notion of o-minimality was developed in the late 1980s after it was noticed that many proofs of analytic and geometric properties of semi-algebraic sets and maps could be carried over verbatim for sub-analytic sets and maps. 
We refer the reader to \cite{Coste2000, TLLoi2010-1, TLLoi2010-2, Dries1998, Dries1996} for the basic properties of o-minimal structures used in this paper. 

\begin{definition}{\rm
An {\em o-minimal structure} on the real field $\mathbb{R}$ is a sequence $\mathcal S := (\mathcal S_n)_{n \in \mathbb{N}}$ such that for each $n \in \mathbb{N}$:
\begin{itemize}
\item [(a)] $\mathcal S_n$ is a Boolean algebra of subsets of $\mathbb{R}^n$.
\item [(b)] If $A \in \mathcal S_m$ and $B \in \mathcal S_n$, then $A \times B \in \mathcal S_{m+n}.$
\item [(c)] If $A \in \mathcal S_{n + 1},$ then $p(A) \in \mathcal S_n,$ where $p \colon \mathbb{R}^{n+1} \to \mathbb{R}^n$ is the projection on the first $n$ coordinates.
\item [(d)] $\mathcal S_n$ contains all algebraic subsets of $\mathbb{R}^n.$
\item [(e)] Each set belonging to $\mathcal S_1$ is a finite union of points and intervals.
\end{itemize}
}\end{definition}

A set $A \subset \mathbb{R}^n$ is said to be a {\em definable set} if $A \in \mathcal{S}_n.$ A map $f \colon A \rightarrow \mathbb{R}^m$ is said to be a {\em definable map} if its graph is definable.

The structure $\mathcal S$ is said to be {\em polynomially bounded} if for every definable function $f \colon \mathbb{R} \rightarrow \mathbb{R},$ there exist $d \in \mathbb{N}$ and $R > 0$ (depending on $f$) such that $|f(x)| \le x^d$ for all $x > R,$ 

Examples of (polynomially bounded) o-minimal structures are
\begin{itemize}
\item the semi-linear sets,
\item the semi-algebraic sets (by the Tarski--Seidenberg theorem),
\item the globally sub-analytic sets, i.e., the sub-analytic sets of $\mathbb{R}^n$ whose (compact) closures in $\mathbb{R}\mathbb{P}^n$ are sub-analytic (using Gabrielov's complement theorem).
\end{itemize}



\subsection{Normals and subdifferentials} 

Here we recall the notions of the normal cones to sets and the subdifferentials of real-valued functions used in this paper. For more details we refer the reader to \cite{Mordukhovich2006,Rockafellar1998}.

\begin{definition}{\rm Consider a set $\Omega\subset\mathbb{R}^n$ and a point ${x} \in \Omega.$
\begin{enumerate}
\item[(i)] The {\em regular normal cone} (known also as the {\em prenormal} or {\em Fr\'echet normal cone}) $\widehat{N}_x \Omega$ to
$\Omega$ at ${x}$ consists of all vectors $v\in\mathbb{R}^n$ satisfying
\begin{eqnarray*}
\langle v, x' - {x} \rangle &\le& o(\|x' -  {x}\|) \quad \textrm{ as } \quad x' \to {x} \quad \textrm{ with } \quad x' \in \Omega.
\end{eqnarray*}

\item[(ii)] The {\em limiting normal cone} (known also as the {\em basic} or {\em Mordukhovich normal cone}) $N_{x} \Omega$ to $\Omega$ at ${x}$ consists of all vectors $v \in \mathbb{R}^n$ such that there are sequences $x^k \to {x}$ with $x^k \in \Omega$ and $v^k \rightarrow v$ with $v^k \in \widehat N_{x^k} \Omega.$
\end{enumerate}
}\end{definition}

If $\Omega$ is a manifold of class $C^1,$ then for every point $x \in \Omega,$ the normal cones $\widehat{N}_{x}\Omega$ and $N_{x} \Omega$ are equal to the normal space to $\Omega$ at ${x}$ in the sense of differential geometry, i.e., $\widehat{N}_x \Omega = {N}_x \Omega$ and $v \perp T_x\Omega$ for all $v \in \widehat{N}_x \Omega,$ where $T_x\Omega$ stands for the {\em tangent space} of $\Omega$ at $x;$ see \cite[Example~6.8]{Rockafellar1998}. 


For a function $f \colon \mathbb{R}^n \rightarrow {\mathbb{R}},$
we define the {\em epigraph} of $f$ to be
\begin{eqnarray*}
\mathrm{epi} f  &:=& \{ (x, y) \in \mathbb{R}^n \times \mathbb{R} \ | \ y \ge f(x) \}.
\end{eqnarray*}
A function $f  \colon \mathbb{R}^n \rightarrow {\mathbb{R}}$ is said to be {\em lower semi-continuous} at ${x}$ if it holds that
\begin{eqnarray*}
\liminf_{x' \to {x}} f(x') &\ge& f({x}).
\end{eqnarray*}

Functional counterparts of normal cones are subdifferentials.
\begin{definition}{\rm
Consider a function $f\colon\mathbb{R}^n \to {\mathbb{R}}$ and a point ${x} \in \mathbb{R}^n.$ The  {\em limiting} and {\em horizon subdifferentials} of $f$ at ${x}$ are defined respectively by
\begin{eqnarray*}
{\partial} f({x}) &:=& \big\{v\in\mathbb{R}^n\;\big|\;(v, -1) \in {N}_{(x, f({x}))} \mathrm{epi} f \big\}, \\
\partial^\infty f({x}) &:=& \big\{v\in\mathbb{R}^n\;\big|\;(v,0)\in N_{({x}, f({x}))} \mathrm{epi} f  \big\}.
\end{eqnarray*}
}\end{definition}

The limiting subdifferential $\partial f({x})$ generalizes the classical notion of gradient. In particular, for $C^1$-smooth functions $f$ on $\mathbb{R}^n,$ the subdifferential consists only of the gradient $\nabla f (x)$ for each $x \in \mathbb{R}^n.$ The horizon subdifferential $\partial^\infty f({x})$ plays an entirely different role--it detects horizontal ``normal'' to the epigraph--and it plays a decisive role in subdifferential calculus; see \cite[Corollary~10.9]{Rockafellar1998} for more details.

\begin{theorem}[Fermat rule]\label{FermatRule}
Consider a lower semi-continuous function $f \colon\mathbb{R}^n \to {\mathbb{R}}$ and a closed set $\Omega \subset \mathbb{R}^n.$ If $\bar{x} \in \Omega$ is a local minimizer of $f$ on $\Omega$ and the qualification condition
\begin{eqnarray*}
\partial^\infty f(\bar{x}) \cap N_{\bar{x}} \Omega &=& \{0\}
\end{eqnarray*}
is valid, then the inclusion $0 \in \partial f(\bar{x}) + N_{\bar{x}} \Omega$ holds.
\end{theorem}

We will also need the following lemma.

\begin{lemma} \label{Lemma210}
Consider a lower semi-continuous definable function $f \colon \mathbb{R}^n \rightarrow {\mathbb{R}}$ and a definable curve $\phi \colon [a, b] \rightarrow \mathbb{R}^n.$ Then for all but finitely many ${{t}} \in [a, b],$ the mappings $\phi$ and $f \circ \phi$ are $C^1$-smooth at ${{t}}$ and satisfy
\begin{eqnarray*}
v \in \partial f (\phi({{t}})) &\Longrightarrow& \langle v, {\phi}'({{t}}) \rangle \ = \ (f \circ \phi)'({{t}}), \\
v \in \partial^\infty f (\phi({{t}})) &\Longrightarrow& \langle v, {\phi}'({{t}}) \rangle \ = \ 0.
\end{eqnarray*}
\end{lemma}

\begin{proof}
(cf. \cite[Proposition~4]{Bolte2007-2} and \cite[Lemma~2.10]{Drusvyatskiy2015}). Without loss of generality, assume that the curve $\phi$ is non-constant. In light of the monotonicity theorem \cite[Theorem~4.1]{Dries1996}, there exists a real number $\epsilon \in (0, 1)$ such that on the open interval $(0, \epsilon)$ we have the mappings $\phi$ and $f \circ \phi$ are $C^1$-smooth and ${\phi}'$ is nonzero. Let
\begin{eqnarray*}
M &:=& \{(\phi({t}), f(\phi({t}))) \ | \ {t} \in (0, \epsilon) \},
\end{eqnarray*}
which is a subset of the epigraph of $f.$ Clearly, $M$ is a connected definable $C^1$-manifold of dimension $1.$ Taking if necessary a smaller $\epsilon,$ we can be sure that there exists a Whitney $C^1$-stratification $\mathscr{W}$ of $\mathrm{epi} f$ such that $M$ is a stratum of $\mathscr{W};$ see \cite[Theorem~4.8]{Dries1996}, for example. 

Take arbitrary (but fixed) ${t} \in (0, \epsilon)$ and $v \in \partial f(\phi({t})).$ By definition, there exist sequences $\{x^k\} \subset U$ and $\{(v^k, t^k)\} \subset \widehat{N}_{(x^k, f(x^k) )} \mathrm{epi} f \subset \mathbb{R}^n \times \mathbb{R},$ such that $x^k \to x := \phi({t})$ and $(v^k, t^k) \to (v, - 1)$ as $k \to \infty.$ Due to the finiteness property of $\mathscr{W},$ we may suppose that the sequence $\{(x^k, f(x^k))\}$ lies entirely in some stratum $S \in \mathscr{W}$ of dimension $d.$ Using the compactness of the Grassmannian manifold of $d$-dimensional subspaces of $\mathbb{R}^n,$ we may assume that the sequence of tangent spaces $T_{(x^k, f(x^k))}S$ converges to some vector space $T$ of dimension $d.$ Then the Whitney-(a) property yields that $T_{(x, f(x))} M \subset T.$ By definition, for each $k \ge 1$ we have that the vector $(v^k, t^k)$ is Fr\'echet normal to the epigraph $\mathrm{epi} f$ of $f$ at $(x^k, f(x^k));$ hence, it is also normal (in the classical sense) to the tangent space  $T_{(x^k, f(x^k))} S.$ By a standard continuity argument, the vector
\begin{eqnarray*}
(v, -1) &=& \lim_{k \to \infty} (v^k, t^k)
\end{eqnarray*}
must be normal to $T$ and a fortiori to $T_{(x, f(x))} M.$ On the other hand, $T_{(x, f(x))} M$ is the vector space generated by the vector
$({\phi}'({t}), (f \circ \phi)'({t})) \in \mathbb{R}^{n} \times \mathbb{R}.$ Consequently, we obtain
\begin{eqnarray*}
\langle v, {\phi}'({t}) \rangle &=& (f \circ \phi)'({t}).
\end{eqnarray*}
A similar argument also shows
\begin{eqnarray*}
\langle v, {\phi}'({{{t}}}) \rangle & = & 0
\end{eqnarray*}
for all ${t} \in (0, \epsilon)$ and all $v \in \partial^\infty(\phi({t})).$

Finally, let $c$ be the supremum of real numbers $T \in [0, 1]$ such that for all but finitely many ${t} \in [0, T),$ we have for all $v \in \partial f (\phi({t}))$ and all $w \in
\partial^\infty(\phi({t})),$
\begin{eqnarray*}
\langle v, {\phi}'({t}) \rangle &=& (f \circ \phi)'({t}) \quad \textrm{ and } \quad \langle w, {\phi}'({{{t}}}) \rangle \ = \ 0.
\end{eqnarray*}
Then $c \ge \epsilon.$ We must prove that $c = 1.$ Suppose that this is not the case. Replacing the interval $[0, 1)$ by the interval $[c, 1)$ and repeating the previous argument, we find a small real number $\epsilon' > 0$ such that for all ${t} \in (c, c + \epsilon'),$ all $v \in \partial f (\phi({t}))$ and all $w \in \partial^\infty(\phi({t})),$
\begin{eqnarray*}
\langle v, {\phi}'({t}) \rangle &=& (f \circ \phi)'({t}) \quad \textrm{ and } \quad \langle w, {\phi}'({{{t}}}) \rangle \ = \ 0,
\end{eqnarray*}
thus contradicting the definition of $c.$ The proof is complete.
\end{proof}

\section{Tangencies} \label{Section3}

Let $f \colon (\mathbb{R}^n, 0) \rightarrow ({\Bbb R}, 0)$ be a continuous definable function germ. Let us begin with the following definition (see also \cite{HaHV2017}).

\begin{definition}{\rm
The {\em tangency variety of $f$ (at $0$)} is defined as follows:
\begin{eqnarray*}
\Gamma(f) &:=& \{x \in (\mathbb{R}^n, 0)  \ | \ \exists \lambda \in {\Bbb R} \textrm{ such that } \lambda x \in \partial f (x) \cup \partial (-f) (x)\}.
\end{eqnarray*}
}\end{definition}

\begin{remark}{\rm
When $f$  is of class $C^1$ one has 
\begin{eqnarray*}
\partial f (x)  & = & - \partial (-f) (x) \ = \ \{\nabla f(x)\},
\end{eqnarray*}
and so 
\begin{eqnarray*}
\Gamma(f) &=& \{x \in (\mathbb{R}^n, 0)  \ | \ \exists \lambda \in {\Bbb R} \textrm{ such that } \lambda x = \nabla f (x)\}.
\end{eqnarray*}
}\end{remark}

By definition, it is not hard to check that $\Gamma(f)$ is a definable set. Moreover, thanks to the Fermat rule (Theorem~\ref{FermatRule}), we can see that for any $t > 0,$ the tangency variety $\Gamma(f)$ contains the set of minimizers (and minimizers) of $f$ on the sphere ${\Bbb S}_{t};$ in particular, $0$ is a cluster point of $\Gamma(f).$

Applying the Hardt triviality theorem (see \cite[Theorem~4.11]{Dries1996}) for the definable function
$$\Gamma(f) \rightarrow \mathbb{R}, \quad x \mapsto \|x\|,$$
we find a constant $\epsilon > 0$ such that the restriction of this function on $\Gamma(f) \cap \mathbb{B}_{\epsilon} \setminus \{0\}$ is a topological trivial fibration. Let $p$ be the  number of connected components of a fiber of this restriction. Then $\Gamma(f) \cap \mathbb{B}_\epsilon \setminus \{0\}$ has exactly $p$ connected components, say $\Gamma_1, \ldots, \Gamma_p,$ and each such component is a definable set. Moreover, for all $t \in (0, \epsilon)$ and all $k = 1, \ldots, p,$ the sets $\Gamma_k \cap \mathbb{S}_t$ are connected. Corresponding to each $\Gamma_k,$ let
$$f_k \colon (0, \epsilon) \rightarrow \mathbb{R}, \quad t \mapsto f_k(t),$$
be the function defined by $f_k(t) :=  f(x),$ where $x \in \Gamma_k \cap \mathbb{S}_t.$

\begin{lemma}
For each $\epsilon > 0$ small enough, all the functions $f_k$ are well-defined and definable.
\end{lemma}

\begin{proof}
Fix $k \in \{1, \ldots, p\}$ and take any $t \in (0, \epsilon).$ We will show that the restriction of $f$ on $\Gamma_k \cap \mathbb{S}_t$ is constant. To see this, let $\phi \colon [0, 1] \rightarrow {\Bbb R}^n$ be a definable $C^1$-curve such that $\phi(\tau) \in \Gamma_k \cap \mathbb{S}_t$ for all $\tau \in [0, 1].$ By definition, we have $\|\phi(\tau) \|  = t$ and either $\lambda(\tau) \phi(\tau) \in \partial f(\phi(\tau))$ or $\lambda(\tau) \phi(\tau) \in \partial (-f)(\phi(\tau))$ for some $\lambda(\tau) \in \mathbb{R}.$ By replacing $f$ by $-f,$ if necessary, we may assume that $\lambda(\tau) \phi(\tau) \in \partial f(\phi(\tau)).$ In view of Lemma~\ref{Lemma210}, for all but finitely many $\tau \in [a, b],$ the mappings $\phi$ and $f \circ \phi$ are $C^1$-smooth at $\tau$ and satisfy
\begin{eqnarray*}
v \in \partial f (\phi(\tau)) &\Longrightarrow& \langle v, {\phi}'(\tau) \rangle \ = \ (f \circ \phi)'(\tau).
\end{eqnarray*}
Therefore
\begin{eqnarray*}
(f \circ \phi)'(\tau) &=&  \langle \lambda(\tau) \phi(\tau)  , {\phi}'(\tau) \rangle \\
&=& \frac{\lambda(\tau)}{2} \frac{d \|\phi(\tau)\|^2}{d\tau}\\
&=& 0.
\end{eqnarray*}
So $f$ is constant on the curve $\phi.$

On the other hand, since the set $\Gamma_k \cap \mathbb{S}_t$ is connected definable, it is path connected. Hence, any two points in $\Gamma_k \cap \mathbb{S}_t$ can be joined by a piecewise $C^1$-smooth definable curve. It follows that  the restriction of $f$ on $\Gamma_k \cap \mathbb{S}_t$ is constant and so the function $f_k$ is well-defined. Finally, by definition, $f_k$ is definable.
\end{proof}


For each $t \in (0, \epsilon),$ the sphere $\mathbb{S}_t$ is a nonempty compact definable set. Hence, the functions
\begin{eqnarray*}
&& \psi \colon (0, \epsilon) \rightarrow \mathbb{R}, \quad t \mapsto \psi(t) := \min_{x \in \mathbb{S}_t} f(x),\\
&& \overline{\psi} \colon (0, \epsilon) \rightarrow \mathbb{R}, \quad t \mapsto \overline{\psi}(t) := \max_{x \in \mathbb{S}_t} f(x),
\end{eqnarray*}
are well-defined and definable. The following lemma is simple but useful.

\begin{lemma} \label{Lemma32}
For $\epsilon > 0$ small enough, the following equalities
\begin{eqnarray*}
\psi(t) &=& \min_{k = 1, \ldots, p} f_k(t) \quad \textrm{ and } \quad \overline{\psi}(t) \ = \ \max_{k = 1, \ldots, p} f_k(t)
\end{eqnarray*}
hold for all $t \in (0, \epsilon).$
\end{lemma}

\begin{proof}
Applying the Curve Selection Lemma (see \cite[Property~1.17]{Dries1996}) and shrinking $\epsilon$ (if necessary), we find a definable $C^1$-curve $\phi \colon (0, \epsilon) \rightarrow \mathbb{R}^n$ such that for all $t \in (0, \epsilon),$
\begin{eqnarray*}
\|\phi(t)\| &=& t \quad \textrm{ and } \quad (f \circ \phi)(t) \ = \ \psi(t).
\end{eqnarray*}
By Lemma~\ref{Lemma210}, then we have for any $t \in (0, \epsilon),$
\begin{eqnarray}
v \in \partial^\infty f(\phi(t)) & \Longrightarrow & \langle v, {\phi}'(t) \rangle \ = \ 0 \nonumber.
\end{eqnarray}
Observe
\begin{eqnarray*}
\langle \phi(t), {\phi}'(t) \rangle & = & \frac{1}{2} \frac{d}{d t} \|\phi(t)\|^2,
\end{eqnarray*}
and hence the qualification condition
\begin{eqnarray*}
\partial^\infty f(\phi(t)) \cap N_{\phi(t)} \mathbb{S}_t = \{0\}
\end{eqnarray*}
holds for all $t \in (0, \epsilon).$ Consequently, since $\phi(t)$ minimizes $f$ subject to $\|x\| = t,$
applying the Fermat rule (Theorem~\ref{FermatRule}), we deduce that $\phi(t)$ belongs to $\Gamma(f).$ Therefore,
\begin{eqnarray*}
\psi(t)  &=& \min_{x \in \mathbb{S}_t} f(x) \ = \ \min_{x \in \Gamma(f) \cap \mathbb{S}_t} f(x) \ = \ \min_{k = 1, \ldots, p} \min_{x \in \Gamma_k \cap \mathbb{S}_t} f(x) \ = \ \min_{k = 1, \ldots, p}  f_k(t).
\end{eqnarray*}

Using the same argument, we also have
\begin{eqnarray*}
\overline{\psi}(t)  &=& \max_{x \in \mathbb{S}_t} f(x) \ = \ \max_{x \in \Gamma(f) \cap \mathbb{S}_t} f(x) \ = \ \max_{k = 1, \ldots, p} \max_{x \in \Gamma_k \cap \mathbb{S}_t} f(x) \ = \ \max_{k = 1, \ldots, p}  f_k(t).
\end{eqnarray*}
The lemma is proved.
\end{proof}

\section{The main result} \label{Section4}

In this section, we fix a polynomially bounded o-minimal structure on $\mathbb{R}.$ The word ``definable'' will mean definable in this structure. 

Let $f \colon (\mathbb{R}^n, 0) \rightarrow ({\Bbb R}, 0)$ be a continuous definable function germ. As in the previous section, we associate to the function $f$ a finite number of (definable) functions $f_1, \ldots, f_p$ of a single variable. Let 
$$K_0 :=\{k \ | \ f_k \textrm{ is constant} \}.$$
By the Growth Dichotomy Lemma (see \cite[Theorem~4.12]{Dries1996}), we can write for each $k \in \{1, \ldots, p\} \setminus K_0,$
\begin{eqnarray*}
f_k(t) &=& a_k t^{\alpha_k} + o(t^{\alpha_k}) \quad \textrm{ as } \quad t \to 0^+,
\end{eqnarray*}
where $a_k \in \mathbb{R}, a_k \ne 0,$ and $\alpha_k \in \mathbb{R}, \alpha_k > 0.$ Put
\begin{eqnarray*}
K_{-} &:= & \{k \notin K_0 \ | \ a_k < 0 \},\\
K_{+} &:= & \{k \notin K_0 \ | \ a_k > 0 \}.
\end{eqnarray*}
Finally we let
$$ \mathrm{Inv} (f) := \left\{ \begin{array}{l l}
(0,  \min_{k \in K_{+}} \alpha_k) \qquad & \text{ if } K_0  \neq \emptyset, K_{-}   =  \emptyset \text{ and } K_{+}   \neq \emptyset,\\
(-\min_{k \in K_{-}} \alpha_k, 0) \qquad & \text{ if } K_0  \neq \emptyset, K_{-}   \neq \emptyset \text{ and } K_{+}   = \emptyset,\\
(-\min_{k \in K_{-}} \alpha_k , \min_{k \in K_{+}} \alpha_k)  & \text{ if }  K_{-}   \neq \emptyset \text{ and } K_{+}   \neq \emptyset,\\
(\min_{k \in K_{+}} \alpha_k, \max_{k \in K_{+}} \alpha_k) & \text{ if } K_0 =  K_{-}   = \emptyset \text{ and }K_{+}  \neq \emptyset.\\
(-\min_{k \in K_{-}} \alpha_k, -\max_{k \in K_{-}} \alpha_k) & \text{ if } K_0 = K_{+}   = \emptyset \text{ and } K_{-}   \neq \emptyset.\\
(0, 0) & \text{ if } K_{-} = K_{+}   = \emptyset.
\end{array} \right. $$
If $\mathrm{Inv}(f) = (a, b),$ we follow the convention that $-\mathrm{Inv}(f)  := \mathrm{Inv}(-f) = (-b, -a).$

We now arrive to the main result of this paper.



\begin{theorem}\label{Theorem41}
Let $f, g \colon (\mathbb{R}^n ,0) \to (\mathbb{R},0)$ be two continuous definable function germs. If $f$ and $g$ are bi-Lipschitz contact equivalent then
\begin{eqnarray*}
\mathrm{Inv} (f) & = & \pm \mathrm{Inv} (g).
\end{eqnarray*}
\end{theorem}

\begin{proof}
Since $f$ and $g$ are bi-Lipschitz contact equivalent, it follows from Theorem~\ref{Theorem1} that there exist a bi-Lipschitz homeomorphism germ $h \colon (\mathbb{R}^n ,0) \to (\mathbb{R}^n,0)$ and some positive constants $c_1, c_2$ and a sign $\sigma \in \{ \pm 1 \}$ such that
\begin{eqnarray} \label{Eqn1}
c_1  f(x) & \leq & \sigma (g \circ h) (x) \ \leq \ c_2  f(x) \quad \text{ for all } \quad \Vert x \Vert \ll 1.
\end{eqnarray}
Assume that $\sigma = 1.$ (The case $\sigma = -1$ is proved similarly.) Consider the definable functions 
\begin{eqnarray*}
\psi_{f} \colon [0, \epsilon) \rightarrow \mathbb{R}, \quad t \mapsto \psi_{f}(t) := \min_{x \in \mathbb{S}_t} f(x), &&
\overline{\psi}_{f} \colon [0, \epsilon) \rightarrow \mathbb{R}, \quad t \mapsto \overline{\psi}_{f}(t) := \max_{x \in \mathbb{S}_t} f(x),\\
\psi_{g} \colon [0, \epsilon) \rightarrow \mathbb{R}, \quad t \mapsto \psi_{g}(t) := \min_{x \in \mathbb{S}_t} g(x),& &
\overline{\psi}_{g} \colon[0, \epsilon) \rightarrow \mathbb{R}, \quad t \mapsto \overline{\psi}_{g}(t) := \max_{x \in \mathbb{S}_t} g(x),
\end{eqnarray*}
where $\epsilon$ is a positive number and small enough so that these functions are either constant or strictly monotone. Assume that we have proved the following relations:
\begin{eqnarray} \label{Eqn2}
\psi_{f} \simeq \psi_{g} \quad \textrm{ and } \quad \overline{\psi}_{f} \simeq \overline{\psi}_{g},
\end{eqnarray}
where $A \simeq B$ means that $A/B$ lies between two positive constants. These, together with Lemma~\ref{Lemma32}, imply easily that $\mathrm{Inv}(f) = \mathrm{Inv}(g),$ which is the desired conclusion.

So we are left with showing~\eqref{Eqn2}. We will prove the first relation; the second one is proved similarly.
Indeed, if $\psi_{f} \equiv 0,$ then $\psi_{g} \equiv 0$ because of \eqref{Eqn1} and there is nothing to prove. So assume that $\psi_{f} \not \equiv 0.$
Since $h$ is a bi-Lipschitz homeomorphism germ, there exists a positive constant $L$ such that 
\begin{eqnarray*}
L^{-1} \| x - x'\| & \leq & \| h(x) - h(x') \| \ \leq L  \ \| x - x' \| \quad \text{ for all } \quad \Vert (x, x') \Vert \ll 1 .
\end{eqnarray*}
In particular, we get
\begin{eqnarray*}
L^{-1} \| x \| & \leq & \| h(x)\| \ \leq \ L \| x \| \quad \text{ for all } \quad \Vert x \Vert \ll 1.
\end{eqnarray*}
This, together with \eqref{Eqn1}, implies that for all sufficiently small $t \geq 0,$
\begin{eqnarray}\label{Eqn3}
c_2 \psi_{f}(t) \ = \ c_2 \min_{ x \in \mathbb{S}_t} f(x)  & \geq &   \min_{ x \in \mathbb{S}_t} (g \circ h)(x) \\
& \geq & \min_{ L^{-1 }t \leq \|h(x) \| \leq L t} (g \circ h)(x) \ = \ \min_{L^{-1 } t \leq \| y \| \leq L t} g(y). \label{Eqn4}
\end{eqnarray}
Let $\phi \colon [0, \epsilon) \rightarrow {\Bbb R}^n$ be a definable curve such that 
\begin{eqnarray*}
g (\phi (t)) & = & \min_{L^{-1 } t \leq \| y \| \leq L t} g(y).
\end{eqnarray*}
Reducing $\epsilon$ if necessary, we may assume that $\phi$ is of class $C^1$ and that either $L^{-1 }t < \| \phi (t) \| < L t$ or $\| \phi (t) \| = L^{-1 }t$ or $\|\phi (t) \| = L t$ for all $t \in [0, \epsilon).$
 
If $L^{-1 }t <  \| \phi (t) \| < L t,$ then $\phi (t)$ is a local minimizer of the function $g$ on the open set $\{ y \in \mathbb{R}^n \, | \, L^{-1 }t < \|y  \| < L t \}.$ By the Fermat rule (Theorem~\ref{FermatRule}), we get $0 \in \partial g(\phi(t))$. This, together with Lemma~\ref{Lemma210}, implies that for all but finitely many $t\in [0, \epsilon),$ 
\begin{eqnarray*}
(g \circ \phi)'(t) &=& \langle 0, \phi'(t) \rangle \ = \ 0.
\end{eqnarray*}
Consequently, $(g \circ \phi)(t) = (g \circ \phi)(0) = 0$ for all $t \in [0, \epsilon),$ which is a contradiction.
 
Therefore, we have $\|\phi (t) \| \equiv r t,$ where either $r = L^{-1}$ or $r = L.$ Moreover, it holds that
\begin{eqnarray*}
\min_{ L^{-1 }t \leq \| y \| \leq L t} g(y)  &=& \min_{y \in \mathbb{S}_{r t}} g(y) \ = \ \psi_{g} (r t ) \ \simeq \ \psi_{g}(t).
\end{eqnarray*}
Combining this with \eqref{Eqn3} and \eqref{Eqn4}, we can find a constant $c > 0$ such that
 \begin{eqnarray*}
c\, \psi_{f} (t) & \geq & \psi_{g} (t) \qquad \text{ for all } \quad 0 \le t \ll 1.
 \end{eqnarray*}
 
Applying the above argument again and using the first inequality in \eqref{Eqn1}, we also obtain
\begin{eqnarray*}
 c' \psi_{g} (t) &\geq &\psi_{f} (t) \qquad \text{ for all } \quad 0 \le t \ll 1
 \end{eqnarray*}
for some $c' > 0.$ Therefore, $\psi_{f} \simeq \psi_{g}.$
\end{proof}

\begin{remark}{\rm
(i) Notice that, in the above proof, we do not assume that the homeomorphism $h$ is definable.

(ii) When $f$ is of class $C^1,$ it is not hard to see that the exponents $\alpha_k$ belong to the set of {\em characteristic exponents} defined by Kurdyka, Mostowski, and  Parusi\'nski \cite{Kurdyka2000}, and moreover, the latter set is preserved by bi-Lipschitz homeomorphisms (see \cite{Henry2004}). On the other hand, we do not know whether the set of the exponents $\alpha_k$ is an invariant of the bi-Lipschitz contact equivalence or not.
}\end{remark}

We conclude the paper with some examples illustrating our results. For simplicity we consider the case where $f$ is a $C^1$-function in two variables $(x, y) \in \mathbb{R}^2.$ By definition, then
\begin{eqnarray*}
\Gamma(f) &:=& \left \{(x, y) \in \mathbb{R}^2 \ | \ y \frac{\partial f}{\partial x} - x \frac{\partial f}{\partial y} = 0 \right\}.
\end{eqnarray*}
In view of Theorem~\ref{Theorem41}, the four functions given below are not bi-Lipschitz contact equivalent to each other.

\begin{example}{\rm
(i) Let $f(x, y) := x^3 + y^6.$ The tangency variety $\Gamma(f)$ is given by the equation:
\begin{eqnarray*}
3x^2y - 6xy^5 &=& 0.
\end{eqnarray*}
Hence, for $\epsilon > 0$ the set $(\Gamma(f)\cap \mathbb{B}_{\epsilon}) \setminus \{ 0 \}$ has six connected components:
\begin{eqnarray*}
\Gamma_{\pm 1} &:=& \left\{ (0, \pm t) \ | \ 0 < t < \epsilon  \right\}, \\
\Gamma_{\pm 2} &:=& \left\{ (2t^4, \pm t) \ | \ 0 < t < \epsilon \right\}, \\
\Gamma_{\pm 3} &:=& \left\{ (\pm t, 0) \ | \ 0 < t < \epsilon \right\}.
\end{eqnarray*}
Consequently,
\begin{eqnarray*}
f|_{\Gamma_{\pm 1}} &=& t^6, \\
f|_{\Gamma_{\pm 2}} &=&  t^6 + 8t^{12}, \\
f|_{\Gamma_{\pm 3}} &=& \pm t^3.
\end{eqnarray*}
It follows that $K_0 = \emptyset, K_{-}  = \{- 3\},K_{+}  = \{\pm 1, \pm 2, 3 \}$ and $\mathrm{Inv} (f) = \{ -3 , 3 \}.$

(ii) Let $f(x, y) := (x^2 -y^3)^2.$ The tangency variety $\Gamma(f)$ is given by the equation:
\begin{eqnarray*}
2yx (3y-2)(x^2 -y^3)&=& 0.
\end{eqnarray*}
Hence, for $0<\epsilon <\frac{2}{3},$ the set $(\Gamma(f)\cap \mathbb{B}_{\epsilon}) \setminus \{ 0 \}$ has six connected components:
\begin{eqnarray*}
\Gamma_{\pm 1} &:=& \left\{ (0, \pm t) \ | \ 0 < t < \epsilon  \right\}, \\
\Gamma_{\pm 2} &:=& \left\{ (\pm t^3, t^2) \ | \ 0 < t < \epsilon \right\}, \\
\Gamma_{\pm 3} &:=& \left\{ (\pm t, 0) \ | \ 0 < t < \epsilon \right\}.
\end{eqnarray*}
Consequently,
\begin{eqnarray*}
f|_{\Gamma_{\pm 1}} &=& t^6, \\
f|_{\Gamma_{\pm 2}} &=& 0, \\
f|_{\Gamma_{\pm 3}} &=& t^4.
\end{eqnarray*}
It follows that $K_0 = \{ \pm 2 \} , K_{-}  = \emptyset,K_{+}  = \{\pm 1, \pm 3 \}$ and $\mathrm{Inv} (f) = \{0, 4\}.$

(iii) Let $f(x, y) := x^2 + y^4.$ The tangency variety $\Gamma(f)$ is given by the equation:
\begin{eqnarray*}
2xy - 4xy^4 &=& 0.
\end{eqnarray*}
Hence, for $0 <\epsilon < \sqrt{\frac{1}{2}},$ the set $(\Gamma(f)\cap \mathbb{B}_{\epsilon}) \setminus \{ 0 \}$ has four connected components:
\begin{eqnarray*}
\Gamma_{\pm 1} &:=& \left\{ (0, \pm t) \ | \ 0 < t < \epsilon  \right\}, \\
\Gamma_{\pm 2} &:=& \left\{ (\pm t, 0) \ | \ 0 < t < \epsilon \right\}.
\end{eqnarray*}
Consequently,
\begin{eqnarray*}
f|_{\Gamma_{\pm 1}} &=& t^4, \\
f|_{\Gamma_{\pm 2}} &=& t^2.
\end{eqnarray*}
It follows that $K_0 =  K_{-}  = \emptyset,K_{+}  = \{\pm 1, \pm 2\}$ and $ \mathrm{Inv} (f) = \{ 2, 4\}.$

(iv) Let $f(x, y) :=-x^2 - 2y^6.$ The tangency variety $\Gamma(f)$ is given by the equation:
\begin{eqnarray*}
-2xy + 6xy^5 &=& 0.
\end{eqnarray*}
Hence, for $0 <\epsilon < \sqrt[4]{\frac{1}{6}},$ the set $(\Gamma(f)\cap \mathbb{B}_{\epsilon}) \setminus \{ 0 \}$ has four connected components:
\begin{eqnarray*}
\Gamma_{\pm 1} &:=& \left\{ (0, \pm t) \ | \ 0 < t < \epsilon  \right\}, \\
\Gamma_{\pm 2} &:=& \left\{ (\pm t, 0) \ | \ 0 < t < \epsilon \right\}.
\end{eqnarray*}
Consequently,
\begin{eqnarray*}
f|_{\Gamma_{\pm 1}} &=& -2t^6, \\
f|_{\Gamma_{\pm 2}} &=& -t^2.
\end{eqnarray*}
It follows that $K_0 =  K_{+}  = \emptyset,K_{-}  = \{\pm 1, \pm 2\}$ and $\mathrm{Inv} (f) = \{ -2, -6\}.$
}\end{example}

\bibliographystyle{abbrv}

\end{document}